\documentclass[a4paper]{article}
\usepackage{amsthm,amsfonts,amsmath,amssymb}
\usepackage[abbrev,nobysame]{amsrefs}
\usepackage[cp1251]{inputenc}
\usepackage[english]{babel}
\usepackage[final]{graphicx}
\usepackage{setspace}
\usepackage[12pt]{extsizes}
\begin{document}
\newcommand{\per}{{\rm per}}
\newcommand{\Per}{{\rm Per}}
\newtheorem{teorema}{Theorem}
\newtheorem{lemma}{Lemma}
\newtheorem{utv}{Proposition}
\newtheorem{svoistvo}{Property}
\newtheorem{sled}{Corollary}
\newtheorem{con}{Conjecture}
\newtheorem{zam}{Remark}
\newtheorem{quest}{Question}

\author{Anna A. Taranenko\thanks{Sobolev Institute of Mathematics, Novosibirsk, Russia, taa@math.nsc.ru}}
\title{Transversals, near transversals, and diagonals in iterated groups and quasigroups}
\date{August 19, 2021}
\maketitle
\begin{abstract}
Given a binary quasigroup $G$ of order $n$,  a $d$-iterated quasigroup $G[d]$ is the $(d+1)$-ary quasigroup equal to the $d$-times composition of $G$ with itself. The Cayley table of every $d$-ary quasigroup is a $d$-dimensional latin hypercube. Transversals and diagonals in multiary quasigroups are defined so as to coincide with those in the corresponding latin hypercube.

We prove that if a group $G$ of order $n$ satisfies the Hall--Paige condition, then the number of transversals in $G[d]$ is equal to $ \frac{n!}{ |G'| n^{n-1}} \cdot n!^{d}  (1 + o(1))$ for large $d$, where $G'$ is the commutator subgroup of $G$.
For a general quasigroup $G$, we obtain similar estimations on the numbers of transversals and near transversals in $G[d]$  and develop a method for counting diagonals of other types in iterated quasigroups.
\end{abstract}



\section{Introduction}

A \textit{latin square} of order $n$ is the Cayley table of a binary quasigroup of order $n$, i.e., an $n \times n$-table  filled by $n$ symbols so that each line (row or column) contains all symbols.  A \textit{transversal} in a latin square of order $n$ is a set of $n$ entries filled by all different symbols such that there is exactly one entry from the set in each row and each column.

In algebraic terms, transversals in binary groups and quasigroups $G$ correspond to complete mappings. A bijection $\varphi: G \rightarrow G$ is called a complete mapping if the mapping $x \mapsto x* \varphi(x)$ is bijective.

The most intriguing conjecture on transversals in latin squares belongs to Ryser~\cite{ryser.conj}.

\begin{con}[Ryser's conjecture]
Every latin square of odd order has a transversal.
\end{con}

Another celebrated conjecture is usually attributed to Brualdi and Stein~\cite{BruRys.combmatth, stein.gentrans}.

\begin{con}[Brualdi's conjecture]
Every latin square has a near transversal. 
\end{con}

Here a \textit{near transversal} in a latin square of order $n$ is a diagonal that contains at least $n-1$ different symbols.

It happens that both of these conjectures are true for groups.  The following statement,  known as the Hall--Paige conjecture, describes all groups with transversals in their Cayley tables.

\begin{teorema}[Hall--Paige conjecture] \label{HPconj}
The Cayley table of a group $G$ has a transversal if and only if every Sylow $2$-subgroup of $G$ is trivial or non-cyclic. 
\end{teorema}

The Hall--Paige conjecture first appeared in~\cite{HallPaige.hyp} and was open for quite a long time. It was proved (modulo the classification of finite simple groups)  in a series of papers of Wilcox~\cite{wilcox.HPred}, Evans~\cite{evans.admsporgr} and Bray~\cite{BCCSZ.hallpaige}. Recently in preprint~\cite{EberManMra.asHallPaige} there was given an alternative (asymptotic) proof for the conjecture. Concerning Brualdi's conjecture, in~\cite{GodHal.grnear} it was proved that the Cayley table of every group contains a near transversal. 

Note that the Hall--Paige conjecture includes the easy observation that the Cayley table of every group of odd order has a transversal. Further, we will say that a group is a \textit{Hall--Paige group} if its Sylow $2$-subgroups are trivial or non-cyclic.

Alongside the existence, it is also interesting to know how many transversals a latin square has. An asymptotic upper bound on the number of transversals was proved in~\cite{my.first} and in~\cite{GlebLur.transup} by another technique. In~\cite{EberManMra.transZn}, Eberhard, Manners, and Mrazovi\'c found the asymptotics of the number of transversals in the  Cayley table of cyclic groups $\mathbb{Z}_n$ of odd order $n$, later Eberhard established it for arbitrary (iterated) abelian groups~\cite{eberhard.moreaddtrip}, and recently these researchers submitted a preprint~\cite{EberManMra.asHallPaige} with a similar result for all groups.

\begin{teorema}[\cite{EberManMra.asHallPaige}] \label{grnasym}
Let $G$ be a Hall--Paige group of order $n$ and $G'$ be the commutator subgroup in $G$.  Then the number of transversals in the Cayley table of $G$ is
$$ \frac{n!}{|G'| n^{n-1}} \cdot n! (e^{-1/2} + o(1)).$$
\end{teorema}

While studying latin squares, we can increase not only their order but also their dimension.  Let a  $d$-dimensional \textit{latin hypercube}  of order $n$  be a $d$-dimensional array of the same order filled by $n$ symbols so that in each line all symbols are different. Latin hypercubes can be considered as the Cayley tables of $d$-ary quasigroups of order $n$. A \textit{transversal} in a latin hypercube of order $n$ is a collection of $n$ entries hitting each hyperplane exactly once and containing all $n$ different symbols of the hypercube. For more formal definitions of latin hypercubes and their transversals see Section~\ref{defsec}.

Wanless~\cite{wanless.surv} generalized Ryser's conjecture on latin hypercubes and proposed that every latin hypercube of odd dimension or odd order has a transversal. As far as we know,  there is no evidence against the generalization of Brualdi's conjecture on latin hypercubes.

\begin{con} \label{multibrualdi}
Every latin hypercube has a near transversal.
\end{con}

In the present paper we focus on a special class of latin hypercubes corresponding to $d$-iterated quasigroups. Such quasigroups have applications in cryptography (see e.g.~\cite{gligoroski.streamchip}),  and their transversals were previously studied in paper~\cite{my.iter} of the present author.

Given  a binary quasigroup $G$ of order $n$ with the operation $*$,  define the \textit{$d$-iterated quasigroup $G[d]$} to be the $(d+1)$-ary quasigroup of order $n$ such that
$$ G[d] (x_1, \ldots, x_{d+1}) = (\ldots((x_1 * x_2) * x_3)* \ldots * x_{d})*x_{d+1}.$$

There were several results on transversals in iterated groups and quasigroups of small order before. The numbers of transversals in both $d$-iterated groups of order $4$ were calculated in~\cite{my.quasi}, and the numbers of transversals in $d$-dimensional latin hypercubes of orders $2$ and $3$ were found in~\cite{my.obzor}. Moreover, in~\cite[Theorem 26]{my.obzor} it was proved that for all even $d$ every $d$-iterated quasigroup has a transversal. One of the main results of~\cite{my.iter} states that for every binary quasigroup $G$ of order $n$ there is  $c = c(G)$ such that the number of transversals in the $d$-iterated quasigroup $G[d]$ (if it is nonzero) asymptotically equals $c n!^{d} (1 + o(1))$. 

The present paper significantly refines the technique and results from~\cite{my.iter} and gives the exact asymptotic of the number of transversals in $d$-iterated groups.

\begin{teorema} \label{grouptransintro}
Let $G$ be a group of order $n$ and $G'$ be the commutator subgroup of $G$. Then the following hold.
\begin{itemize}
\item  If $G$ satisfies the Hall--Paige condition, then for all $d$, the $d$-iterated group $G[d]$ has  a transversal.
\item If $G$ does not satisfy the Hall--Paige condition, then  $G[d]$ has a transversal if $d$ is even and does not have a transversal if $d$ is odd.
\end{itemize}
If the number $T(d)$ of transversals in $G[d]$ is nonzero, then 
 $$T (d) = \frac{n! }{|G'|n^{n-1}}   \cdot n!^d (1 + o(1)) $$
 as $d \rightarrow \infty$.
\end{teorema}

It is interesting to note that Theorems~\ref{grnasym} and~\ref{grouptransintro} give similar asymptotics for the number of transversals. It seems probable that  Theorem~\ref{grnasym} admits an extension to the case of arbitrary $d$-iterated groups of large order (see~\cite{eberhard.moreaddtrip}).

For a general quasigroup we prove the following result.

\begin{teorema} \label{quasitransintro}
Let $G$ be a binary quasigroup of order $n$.  There is $d_0 \in \mathbb{N}$ such that for all $d \geq d_0$ one of the following possibilties for transversals in iterated quasigroups $G[d]$ occurs:
\begin{itemize}
\item every iterated quasigroup $G[d]$ has a transversal;
\item $G[d]$ has no transversals when $d$ is odd and contains transversals when $d$ is even.
\end{itemize}
There is an integer $r =r(G)$, $1 \leq r \leq n$, such that if the number $T(d)$ of transversals in $G[d]$ is nonzero, then 
 $$T (d) = \frac{n! }{rn^{n-1}}   \cdot n!^d(1 + o(1)) $$
 as $d \rightarrow \infty$. Moreover, if $G$ is a loop, then  $r \leq |G'|$, where $G'$ is the commutator subloop of $G$.
\end{teorema}

Theorems~\ref{grouptransintro} and~\ref{quasitransintro} imply that iterated abelian groups have the asymptotically maximal number of transversals among all iterated quasigroups, that answers one of the questions from~\cite{my.iter}.

The method developed in this paper allows us to count not only transversals but other types of diagonals and structures in iterated quasigroups.
For instance, we prove the following result for near transversals being in line with Conjecture~\ref{multibrualdi}.

\begin{teorema} \label{nearintro}
Let $G$ be a binary quasigroup of order $n$.
\begin{enumerate}
\item There is $d_0 \in \mathbb{N}$ such that for all $d \geq d_0$ the $d$-iterated quasigroup $G$ has a near transversal. Moreover, if Brualdi's conjecture is true, then all $d$-iterated quasigroups $G[d]$ have a near transversal.
\item There is an integer $r =r(G)$, $1 \leq r \leq n$, such that the number $N(d)$ of near  transversals  in $G[d]$ is 
 $$N (d)  = c(G,d)  \frac{n!}{ rn^{n-1}} \cdot n!^{d}  (1 + o(1)) $$
 as $d \rightarrow \infty$, where $c(G,d) = \frac{n}{2} (r-1) +1$ when $G[d]$ has a transversal, and  $c(G,d) = \frac{n}{2} r$ otherwise. If $G$ is a group, then $r = |G'|$.
\end{enumerate}
\end{teorema}

Recall that we treat a near transversal in a latin hypercube of order $n$ as a set of $n$ entries  filled by $n$ or $n-1$ different symbols.

\section{Main definitions and preliminaries} \label{defsec}

In what follows, $\mathcal{I}_n$ stands for the set $\{ 1, \ldots, n\}$ and $\mathcal{I}_n^k$ is used for the set of all $k$-tuples with entries from $\mathcal{I}_n$.   

An $n$-tuple $W \in \mathcal{I}_n^n$, $W = (w_1, \ldots, w_n)$, with different entries $w_i$ in all positions is said to be a \textit{permutation}.  Let $\mathcal{W}$ denote the set of all permutations and $\mathbb{W}$ be the identity permutation $(1, \ldots, n)$. Let us denote by $\textbf{a}$ the $n$-tuple from $\mathcal{I}_n^n$ all of whose entries are equal to $a$. Given a tuple $V \in \mathcal{I}_n^n$, we use $V_i(b)$ for a tuple from $\mathcal{I}_n^n$ that coincides with $V$ in all positions except, possibly, the $i$-th position, in which  $V_i(b)$ equals $b$. For a permutation $\pi \in S_n$ and a tuple $V \in \mathcal{I}_n^n$, $V = (v_1, \ldots,v_n )$, let $\pi(V)  = (\pi(v_1), \ldots,\pi(v_n) )$ and  $V^\pi = (v_{\pi(1)}, \ldots, v_{\pi(n)})$.

A \textit{binary quasigroup} $G$ of order $n$ is defined by a binary operation $*$ over a set $\mathcal{I}_n$ satisfying the following condition: for each $a_0, a_1, a_2 \in \mathcal{I}_n$ there exist unique $x_1, x_2 \in \mathcal{I}_n$ such that both $a_0 = a_1 * x_2$ and $a_0 = x_1 * a_2$ hold.  A \textit{$d$-ary quasigroup} $f$ of order $n$ is a function $f: \mathcal{I}_n^d \rightarrow \mathcal{I}_n$ such that the equation $x_0 = f(x_1, \ldots, x_n)$ has a unique solution for any one variable if all the other $n$ variables are specified arbitrarily. 
 
A \textit{composition}  of a $d$-ary quasigroup $f$ and a $k$-ary quasigroup $g$ of orders $n$  is the $(d+k-1)$-ary quasigroup  $h$ of order $n$ such that for some permutation $\sigma \in S_{d+k}$ it holds
$$h(x_1, \ldots, x_{d+k-1}) = x_{d+k} \Leftrightarrow g(x_{\sigma(1)}, \ldots, x_{\sigma(k)}) = f(x_{\sigma(k+1)},   \ldots, x_{\sigma(d+k)} ) . $$
 
In this paper we mostly will work with the composition of quasigroups $f$ and $g$ defined as 
$$h(x_1, \ldots, x_{d+k-1}) = x_{d+k} \Leftrightarrow f(g(x_1, \ldots, x_k), x_{k+1}, \ldots, x_{d+k-1} ) = x_{d+k}. $$

Given a binary quasigroup $G$ of order $n$ with the quasigroup operation $*$, the \textit{$d$-iterated quasigroup $G[d]$} is the $(d+1)$-ary quasigroup of order $n$ obtained as the composition of $d$ copies of  the quasigroup $G$ with itself:
$$G[d] (x_1, \ldots, x_{d+1}) = x_{d+2} \Leftrightarrow (\ldots((x_1 * x_2) * x_3)* \ldots * x_{d})*x_{d+1} = x_{d+2}.$$

In particular, the $0$-iterated quasigroup $G[0]$ is the identity $1$-ary mapping ($G[0](x) = x$ for all $x$), and the $1$-iterated quasigroup $G[1]$ coincides with  the binary quasigroup $G$.
 
A \textit{$d$-dimensional latin hypercube $Q$ of order $n$} is the Cayley table of a $d$-ary quasigroup of the same order.  Equivalently, a $d$-dimensional latin hypercube of order $n$ is an array indexed by elements from $\mathcal{I}_n^d$, whose entries take values from the set  $\mathcal{I}_n$ so that in each line ($1$-dimensional plane) of the array all $n$ symbols occur.  In what follows, we identify a $d$-ary quasigroup and the corresponding $d$-dimensional latin hypercube.

Every $d$-ary quasigroup $f = f(x_1, \ldots, x_d)$ of order $n$  can be considered as an imaging of the first coordinate $x_1$ by the action of all other coordinates $x_2, \ldots, x_d$. Similarly, for any $k$-tuple $U$ from $\mathcal{I}_n^k$ we can find an image $V$ of  the tuple $U$ in  the $d$-ary quasigroup $f$ by the action of $k$-tuples $W_{1}, \ldots, W_{d-1} \in \mathcal{I}_n^k$ from  the relation $f(U, W_1, \ldots, W_{d-1}) = V$ satisfied entrywise. Let us use this approach to define a collection of elements in a quasigroup which are close to being diagonals.

Given a  $d$-ary quasigroup $f$, define an \textit{$U$-diagonal of type $V$} to be a collection of permutations $(W_1, \ldots, W_{d-1})$, $W_i \in \mathcal{W}$, for which the equality $$f(U, W_1, \ldots, W_{d-1}) = V$$ holds entrywise.  We define a \textit{transversal} in a quasigroup $f$ to be an arbitrary $\mathbb{W}$-diagonal of type $W$, where $W$ is a  permutation. 

In the definition of $U$-diagonals of type $V$, we require that the middle tuples $W_1, \ldots, W_{d-1}$ are permutations only because in the framework of this study we are interested in transversals and diagonals.   For other structures in latin hypercubes, one can take the tuples  $W_1, \ldots, W_{d-1}$ from any other appropriate class.

Given a $d$-ary quasigroup $f$ of order $n$, define the \textit{transition matrix} $T$ to be the matrix of order $n^n$ with entries $t_{U,V}$, $U, V \in \mathcal{I}_n^n$, equal to the number of $U$-diagonals of type $V$ in the quasigroup $f$. Note that for every $d$-ary quasigroup of order $n$, the transition matrix $T$ is an integer nonnegative matrix with row and column sums equal to $n!^{d-1}$.  

For an illustration of the introduced concepts, consider the following simple example.

\textbf{Example 1.} Let $G$ be a binary quasigroup (and group) of order $2$ with the Cayley table
$$\begin{array}{c|cc}
*  & 1 & 2 \\
\hline
1 & 1 & 2 \\ 2 & 2 & 1
\end{array}$$
The set $\mathcal{I}_2^2$ consists of four tuples:
$$ (1,1), ~~~ (1,2), ~~~ (2,1), ~~~ (2,2),$$
with two of them being permutations: $\textbf{(1,2)}$, $\textit{(2,1)}$.  The transition matrix $T$ of the quasigroup $G$ is the following matrix of order $4$, in which a diagonal given by the first permutation is highlighted in bold text and the other is highlighted in italics:
$$\begin{array}{r|cccc}
~ & (1,1) & (1,2) & (2,1) & (2,2) \\
\hline
(1,1) & 0 & \textbf{1} & \textit{1} & 0 \\ 
(1,2) & \textbf{1}  & 0 & 0 & \textit{1} \\
(2,1) & \textit{1} & 0 & 0 & \textbf{1} \\ 
(2,2) & 0 & \textit{1}  & \textbf{1}  & 0
\end{array}$$

Let us establish several properties of transition matrices. Firstly, we consider the transition matrices of isotopic binary quasigroups.

Binary quasigroups $(G, *)$ and $(H, \cdot)$ are called \textit{isotopic} if there are bijections $\alpha, \beta, \gamma: G \rightarrow H$ such that $\alpha(x) \cdot \beta(y) = \gamma(x * y)$ for all $x,y \in G$. Appealing to the terms of latin squares, we will say that $\alpha$ is a \textit{row isotopy}, $\beta$ is a \textit{column isotopy}, and $\gamma$ is a \textit{symbol isotopy}.

\begin{lemma} \label{tranitiso}
Let $(G, *)$ and $(H, \cdot)$  be binary quasigroups of order $n$ with transition matrices $T = (t_{U,V})$ and $R = (r_{U,V})$, respectively.
\begin{enumerate}
\item  If there is a row isotopy $\alpha$ between $G$ and $H$, then  $ t_{U,V} = r_{\alpha(U),V} $ for all $U,V \in \mathcal{I}_n^n$.
\item  If there is a column isotopy $\beta$ between $G$ and $H$, then $T = R$.
\item  If there is a symbol isotopy $\gamma$ between $G$ and $H$, then $ t_{U,V} =r_{U,\gamma(V)} $ for all $U,V \in \mathcal{I}_n^n$.
\end{enumerate}
\end{lemma}

\begin{proof}
Let a permutation $W$ define a $U$-diagonal of type $V$ in the quasigroup $G$, that is $U * W = V$.

1.  Since $ x * y = \alpha(x) \cdot y $,  the permutation $W$ gives an $\alpha(U)$-diagonal of type $V$ in the quasigroup $H$: $\alpha(U) \cdot W = V$.

2.  Since $ x * y = x \cdot \beta(y) $,  the permutation $\beta(W)$ gives the same $U$-diagonal of type $V$ in the quasigroup $H$: $U \cdot \beta(W) = V$.

3.  Since $ \gamma(x * y) = x \cdot y $,  the permutation $W$ gives a $U$-diagonal of type $\gamma (V)$ in the quasigroup $H$: $U \cdot W = \gamma(V)$.
\end{proof}

Next, we state the following key property of the transition matrix of a $d$-iterated quasigroup.

\begin{lemma} \label{diagiter} 
Given a binary quasigroup $G$ of order $n$ with the transition matrix $T= (t_{U,V})$, the transition matrix of the $d$-iterated quasigroup $G[d]$ is $T^{d}$.  In particular, if there is a $U$-diagonal of type $Z$ in $G[d]$ and a $Z$-diagonal of type $V$ in $G[k]$ then there is a $U$-diagonal of type $V$ in $G[d+k]$.
\end{lemma}

\begin{proof}
The proof is by induction on $d$. For $d= 0$ and $d=1$ the statement of the lemma follows from definitions.

Assume that the $(d-1)$-iterated quasigroup $G[d-1]$ has the transition matrix $R = T^{d-1}$, $R = (r_{U,V})$.  To prove the induction step, we note that a collection of permutations $(W_1, \ldots, W_{d-1}, W_d)$ is a $U$-diagonal of type $V$ in $G[d]$  if and only if   for some tuple $Z \in \mathcal{I}_n^n$ the collection $(W_1, \ldots, W_{d-1})$ is a $U$-diagonal of type $Z$ in the quasigroup $G[d-1]$ and the permutation $W_d$ is a $Z$-diagonal of type $V$  in $G$. So the number of all $U$-diagonals of type $V$ in  $G[d]$ is equal to $\sum\limits_{Z \in \mathcal{I}_n^n } r_{U,Z} \cdot t_{Z,V}$.
\end{proof}

Summing up, we see that the problem of the asymptotics of the numbers of $U$-diagonals of type $V$ in iterated quasigroups $G[d]$ is equivalent to the question of the asymptotic behavior of the powers of the transition matrices $T$ of $G$.

\subsection{Perron--Frobenius theory, equivalence classes, and units} 

To study the behavior of powers of a transition matrix $T$, we use some results of  Perron--Frobenius theory.

A matrix $A = (a_{i,j})$ is said to be \textit{nonnegative} if all $a_{i,j} \geq 0$.
A nonnegative matrix is called \textit{doubly stochastic} if the sum of entries of $A$ in each row and column equals $1$. Let $J_n$ denote the doubly stochastic matrix of order $n$, whose entries all equal $1/n$.

A nonnegative matrix $A$ is called \textit{irreducible} if for each pair of indices $(i,j)$ there is $l \in \mathbb{N}$ such that the $(i,j)$-th entry of $A^l$ is positive. The \textit{period} of an irreducible matrix can be defined as the greatest common divisor of all $l$  for which the $(i,i)$-th entries of $A^l$ are positive for all $i$.

The following property  can be found in~\cite{PerMir.specdoubstoch} or it can be easily derived from definitions.

\begin{teorema}[\cite{PerMir.specdoubstoch}]  \label{PFform}
For every doubly stochastic matrix $A$  there is a permutation matrix $P$ such that 
$$PAP^{-1} = \left(  \begin{array}{ccc}  
B_1 & ~ & 0 \\
~ & \ddots & ~ \\
0 & ~ & B_k\
\end{array}\right),$$
where $B_i$ are irreducible doubly stochastic matrices. 
\end{teorema}

From the Perron--Frobenius theory, we have the following result for  irreducible doubly stochastic matrices.

\begin{teorema} \label{PerFrob}
Let $A$ be an irreducible doubly stochastic matrix of order $n$ and period $\tau$.  After  appropriate simultaneous permutations of rows and columns, we have the following limits for powers of $A$:
$$\lim_{k \rightarrow \infty} A^{k  \tau} =
\left(  \begin{array}{ccc}  
J_{n/ \tau} &  ~ & 0 \\
~ & \ddots & ~ \\
0 & ~ & J_{n/ \tau} \\
\end{array}\right); 
$$
$$
\lim_{k \rightarrow \infty} A^{k  \tau +1} =
\left(  \begin{array}{cccc} 
0 &  J_{n/ \tau} &  ~ & 0 \\
\vdots   & ~ & \ddots & ~ \\
0  & 0 & ~  & J_{n/ \tau} \\
J_{n/ \tau}  & 0 & \cdots & 0\\
\end{array}\right);  ~ \cdots
 $$
 $$
\lim_{k \rightarrow \infty} A^{k  \tau + \tau-1} =
\left(  \begin{array}{cccc} 
0 & \cdots  &  0 & J_{n/ \tau} \\
J_{n/ \tau}   & ~ & 0  & 0 \\
~  & \ddots & ~  & \vdots \\
 0 & ~ & J_{n/ \tau} & 0\\
\end{array}\right).
 $$
 Moreover, for all $i,j \in \{1, \ldots, n\}$ there is $r \in \{0, \ldots, \tau-1 \}$ such that the $(i,j)$-entry of the $t$-power of $A$ is nonzero only if $t \equiv r \mod \tau$. 
\end{teorema}

Note that, for a given binary quasigroup $G$ of order $n$ with transition matrix $T$, the matrix $n!^{-1} T$ is a doubly stochastic matrix of order $n^n$.  
By Theorem~\ref{PFform}, there is a  permutation matrix $P$  for which
$$PTP^{-1} = \left(  \begin{array}{ccc}  
B_1 & ~ & 0 \\
~ & \ddots & ~ \\
0 & ~ & B_{m} \\
\end{array}\right),$$
where $B_i$ are some irreducible doubly stochastic matrices.  Next, there is a simultaneous permutation of rows and columns of  $PTP^{-1}$ that preserves its block-diagonal form and puts each block $B_i$ into a block matrix whose limits and structure are given by Theorem~\ref{PerFrob}. In what follows, we assume everywhere that the transition matrix $T$ has the described block-diagonal form.

For a given binary quasigroup $G$, we divide  the set of all $n$-tuples $\mathcal{I}_n^n$  into $m = m(G)$ \textit{equivalence classes} $\mathcal{U}_1, \ldots, \mathcal{U}_{m}$ such that each equivalence class $\mathcal{U}_i$ is exactly the set of rows (or columns) of  the block $B_i$ of the transition matrix $T$. In particular, tuples $U$ and $V$ belong to the same equivalence class $\mathcal{U}_i$ if and only if there exists a $U$-diagonal of type $V$ in $G[d]$ for some $d$.  Let the \textit{period} $\tau_i$ of the equivalence class $\mathcal{U}_i$ be the period of the irreducible block $B_i$.

By Theorem~\ref{PerFrob} and the construction of matrix $T$, we divide each equivalence class $\mathcal{U}_i$ into $\tau_i$ subsets $\mathcal{Y}_1^i, \ldots ,\mathcal{Y}_{\tau_i}^i$ of equal sizes defined by the following property: given tuples $U \in \mathcal{Y}_k^i$, $V \in \mathcal{Y}_l^i$, there is a $U$-diagonal of type $V$ in the $d$-iterated quasigroup $G[d]$ only if  $l - k \equiv d \mod \tau_i$. Let us call such subsets $\mathcal{Y}_1^i, \ldots ,\mathcal{Y}_{\tau_i}^i$  of the equivalence class $\mathcal{U}_i$ by  \textit{units}. Roughly speaking, units $\mathcal{Y}_1^i, \ldots ,\mathcal{Y}_{\tau_i}^i$ determine the block structure of the block $B_i$ of the transition matrix $T$. 

For the binary quasigroup $G$ of order $2$ from Example 1, the set of all tuples $\mathcal{I}_2^2$ composes the single equivalence class $\mathcal{U}_1$, which  consists of two units:
$$\mathcal{Y}^1_1 = \{ (1,2), (2,1)\}, ~~~~ \mathcal{Y}^1_2 = \{ (1,1), (2,2)\}.$$

\section{Diagonals in iterated quasigroups}

By Theorem~\ref{PerFrob}, the question on the asymptotic behavior of the number of $U$-diagonals of type $V$ in a $d$-iterated quasigroup $G[d]$ is reduced to studying of the properties and sizes of equivalence classes and units produced by the transition matrix $T$.  In the present section, we carry out this research.

To distinguish equivalence classes  and their units, we use  $n$-tuples $\textbf{a}_i (b)$, whose entries all equal $a$, except the $i$-th position, which is  equal to $b$. We stress that this notation includes the possibility that $b=a$.

We start this section with the fact that every equivalence class contains at least one tuple of the form $\textbf{a}_i(b)$.

\begin{utv} \label{toconst}
Let $G$ be a binary quasigroup of order $n$. Then for each $n$-tuple $V \in \mathcal{I}_n^n$ and $a,i \in \mathcal{I}_n$ there is $b \in \mathcal{I}_n$ and $k \in \mathbb{N}$ such that  there exists a $V$-diagonal of type  $\textbf{a}_i(b)$ in $G[2k]$.
\end{utv}

\begin{proof}
The proof is by induction on the number of positions in which tuple $V$ is distinct from tuples $\textbf{a}_i(b)$, $b \in \mathcal{I}_n$. The base of induction is  $V = \textbf{a}_i(b)$ for some $b \in \mathcal{I}_n$. In this case the statement is trivially true because  there is a  $V$-diagonal of type $\textbf{a}_i(b)$ in the $0$-iterated quasigroup $G[0]$.

Assume that the tuple $V$ is different from some tuples $\textbf{a}_i(b)$ in at most $k$ positions, $k \geq 1$.
Without loss of generality, suppose that $i$ equals $n$, the tuple $V$ equals $a$ in the first $n-k -1$ positions and it is different from $a$ in positions $n-k, \ldots, n$: $V = (a, \ldots, a, v_{n-k}, \ldots,  v_{n})$. 

Consider permutations $W, W' \in \mathcal{W}$, $W = (w_1, \ldots, w_n)$, $W' = (w'_1, \ldots, w'_n)$, such that the following equalities hold
\begin{gather*} 
(a * w_1) * w'_1 = a; \\
\vdots \\
(a * w_{n-k-1}) * w'_{n-k-1} = a; \\
(v_{n-k} * w_{n-k}) * w'_{n-k} = a.
\end{gather*} 
In other words, we demand that $(V*W) *W' = V'$ for some $n$-tuple $V'$, whose first $n-k$ positions are equal to $a$. 

Let us show that such permutations $W$ and $W'$ exist.  Firstly, we chose their entries $w_{n-k}$ and $w'_{n-k}$  so that the equality $(v_{n-k} * w_{n-k}) * w'_{n-k} = a$ holds. Note that for every $a \in \mathcal{I}_n$ there are exactly $n$ different pairs $(w,w')$ satisfying the equality $(a * w) *w' = a$. The choice of $w_{n-k}$ and $w'_{n-k} $ spoils no more than two of these pairs, so we are able to find $n-k-1$ different pairs $(w_1,w'_{1}), \ldots ,(w_{n-k-1}, w'_{n-k-1})$ for the first components of permutations $W$ and $W'$. All other components of permutations $W$ and $W'$ are arbitrary.

Thus we have found a tuple $V'$ different from tuples  $\textbf{a}_i(b)$ in at most $k-1$ components and for which there is a $V$-diagonal of type $V'$ in $G[2]$. By the assumption of induction, there is a $V'$-diagonal of type $\textbf{a}_i(b)$ in $G[2k-2]$. Therefore, by Lemma~\ref{diagiter}, one can find a $V$-diagonal of type $\textbf{a}_i(b)$ in $G[2k]$.
\end{proof}

To estimate the sizes and numbers of equivalence classes and their units, we prove the following proposition.

\begin{utv} \label{canoninunits}
Let $G$ be a binary quasigroup of order $n$. Then for given $a,j \in \mathcal{I}_n$ every unit  contains at least one of the tuples $\textbf{a}_j(b)$, $b \in \mathcal{I}_n$. Moreover, for a given equivalence class $\mathcal{U}_i$, each unit $\mathcal{Y}_k^i$ from this class contains  the same number of tuples $\textbf{a}_j(b)$.
\end{utv}

\begin{proof}
By Proposition~\ref{toconst}, for given $a,j \in \mathcal{I}_n$ each equivalence class $\mathcal{U}_i$ contains at least one of the tuples $\textbf{a}_j(b)$, $b \in \mathcal{I}_n$.  

Let $\mathcal{Y}_1^i, \ldots, \mathcal{Y}_{\tau_i}^i$ be the units in an equivalence class $\mathcal{U}_i$.  Note that one can always find a pair of permutations $W, W' \in \mathcal{W}$ such that $(\textbf{a}_j(b) * W) * W' = \textbf{a}_j(b')$ for some $b' \in \mathcal{I}_n$. By definition of units and since $\mathcal{Y}_1^i, \ldots, \mathcal{Y}_{\tau_i}^i$ is a partition of the equivalence class $\mathcal{U}_i$,  if $\textbf{a}_j(b)$ belongs to some unit $\mathcal{Y}^i_k$, then the unit $\mathcal{Y}^i_{l}$ with $l \equiv k +2 \mod \tau_i$ contains some tuple $\textbf{a}_j(b')$.  From this and Proposition~\ref{toconst} we deduce that all units within an equivalence class $\mathcal{U}_i$ contain tuples of the form $\textbf{a}_j(b)$, $b \in \mathcal{I}_n$.

Assume that $\textbf{a}_j(b_1), \ldots, \textbf{a}_j(b_l)$  are all tuples of such a form in a unit $\mathcal{Y}_k^i$, and that $\textbf{a}_j (c_1)$ is from any other unit $\mathcal{Y}^i_{k'}$. By definition of units and equivalence classes,  there is some $d \in \mathbb{N}$ for which there exists an $\textbf{a}_j (b_1)$-diagonal of type $\textbf{a}_j(c_1)$ in $G[d]$ corresponding to a collection of permutations  $(W_1, \ldots, W_d)$. Note the same collection of permutations gives $\textbf{a}_j (b_s)$-diagonals of type $\textbf{a}_j(c_s)$ in $G[d]$, $s = 1, \ldots, l$, with all  $c_1, \ldots, c_l$ being  distinct. So each unit $\mathcal{Y}_{k'}^i$ in the class $\mathcal{U}_i$ contains the same number of tuples of the form $\textbf{a}_j(b)$, $b \in \mathcal{I}_n$, as the unit $\mathcal{Y}_k^i$.
\end{proof}

As a trivial corollary of Proposition~\ref{canoninunits}, we see that the total number of all units across all equivalence classes is not greater than the order $n$ of $G$.

Recall that for a given $n$-tuple $V$ we use $V_i(b)$ for denoting the $n$-tuple that coincides with $V$ in all but the $i$-th position and equals $b$ in the $i$-th position. Let  $V_i^*$ be the set  $\{ V_i(1), \ldots, V_i(n) \}$.

\begin{utv} \label{Vinunits}
Let $G$ be a binary quasigroup of order $n$ and $\mathcal{Y}_1^i, \ldots, \mathcal{Y}^i_{\tau_i}$ be the units in an equivalence class $\mathcal{U}_i$. Then for all $V \in \mathcal{I}_n^n$ and $j \in \mathcal{I}_n$ each   unit $\mathcal{Y}^i_k$ contains the  same number $r_i$ of tuples from the set $V_j^*$.
\end{utv}

\begin{proof}
Proposition~\ref{toconst} and definition of units imply that each unit $\mathcal{Y}_k^i$ in an equivalence class $\mathcal{U}_i$ contains at least one tuple from the set $V^*_j$. Acting similar to the proof of Proposition~\ref{canoninunits}, we see that the number of tuples from $V^*_j$ in $\mathcal{Y}_k^i$  coincides with the number of canonical tuples $\textbf{a}_j(b)$ in this unit. Recall that, by Theorem~\ref{PerFrob}, each unit within an equivalence class $\mathcal{U}_i$ has the same size, so units $\mathcal{Y}_k^i$  contain the same number $r_i$ of tuples from $V^*_j$. 
\end{proof}

Note that Proposition~\ref{Vinunits} is equivalent to the following fact: for all tuples $U, V \in \mathcal{I}_n^n$ and $i \in \mathcal{I}_n$,  every $d$-iterated quasigroup $G[d]$ with large enough $d$ has a $U$-diagonal of type $V_i(a)$ for an appropriate $a \in \mathcal{I}_n$.
As a simple corollary of Proposition~\ref{Vinunits}, we have the below estimation on the sizes of units.

\begin{sled}\label{unitsizes}
Let $G$ be a binary quasigroup of order $n$.
There exist integers $r_i$, $1 \leq r_i \leq n$ such that  each unit in an equivalence  class $\mathcal{U}_i$ has cardinality $r_i  n^{n-1}$.
\end{sled}

Using similar ideas, let us prove that units and equivalence classes are closed under permutations of positions.

\begin{utv} \label{permclosed}
Given a binary quasigroup $G$ of order $n$, if a tuple $V \in \mathcal{I}_n^n$ belongs to a unit $\mathcal{Y}$ of an equivalence class $\mathcal{U}$, then every tuple $V^\pi$, whose coordinates are permuted by $\pi \in S_n$, belongs to $\mathcal{Y}$. 
\end{utv}

\begin{proof}
It is sufficient to show that there exists a set of tuples $S \subset \mathcal{Y}$  such that for all $U \in S$ and permutations $\pi \in S_n$ the tuple $U^\pi$ belongs to $S$. Indeed, for every tuple  $V \in \mathcal{Y}$, we can find a tuple $U \in S$ such that  there exists a $U$-diagonal  of type $V$ in some $G[d]$ given by a collection of permutations $(W_1, \ldots, W_d)$. Applying  $\pi$ to coordinates of tuples $U$, $V$, and $W_i$, $i = 1, \ldots, d$, we obtain an $U^{\pi}$-diagonal of type $V^\pi$ in $G[d]$. Since $U^\pi$ belongs to $\mathcal{Y}$, the tuple $V^\pi$ is also from the unit $\mathcal{Y}$.

Let us find the required set $S$. By Proposition~\ref{canoninunits}, every unit $\mathcal{Y}$ has a tuple $\textbf{a}_j(b)$ for any $a, j \in \mathcal{I}_n$ and certain $b \in \mathcal{I}_n$. Let $r$ be the number of tuples of the form $\textbf{a}_j(b)$ in $\mathcal{Y}$ for given $a, j \in \mathcal{I}_n$. By definitions of units, there are collections of permutations $D_i = (W_1^i, \ldots, W_d^i)$ corresponding to $\textbf{a}_1(b_i)$-diagonals of types $\textbf{a}_2(c_i)$. Here $\textbf{a}_1(b_i)$ and $\textbf{a}_2(c_i)$ are all tuples of such a form in the unit $\mathcal{Y}$, $i = 1, \ldots, r$.  For every $k\in \mathcal{I}_n$, $k \neq 1$, there is a permutation of coordinates $\pi_k \in S_n$ such that the collections $D_i^{\pi_k}$ correspond to  $\textbf{a}_1(b_i)$-diagonals of types $\textbf{a}_k(c_i)$.

Since all $c_i$ are different, the sets $\{ \textbf{a}_k (c_1), \ldots, \textbf{a}_k (c_r) \}$ consist of all tuples of the form $\textbf{a}_k(b)$. Using similar reasoning for other coordinates of tuples, we conclude that $S = \{ \textbf{a}_i(c_j)  \}$, $i \in \mathcal{I}_n$, $j = 1, \ldots, r$, is the set of all tuples of such a form in the unit $\mathcal{Y}$.  It is easy to see that $S$ is closed under permutations of coordinates.
\end{proof}

A direct corollary of Proposition~\ref{permclosed} is that all permutations $\mathcal{W}$ are contained in a single unit of some equivalence class. Without loss of generality, let $\mathcal{U}_1$ be the equivalence class including the set $\mathcal{W}$. 

Since we are interested in transversals and  $\mathbb{W}$-diagonals of other types in iterated quasigroups, we pay special attention to the class $\mathcal{U}_1$. Most significantly, the class $\mathcal{U}_1$ has several remarkable properties that other equivalence classes do not possess.

\begin{lemma} \label{transclass}
Let $G$ be a  binary quasigroup of order $n$. Then the period $\tau_1$ of the class $\mathcal{U}_1$ is not greater than $2$. Moreover, all permutations $W \in \mathcal{W}$ and tuples $\textbf{a}$, $a \in \mathcal{I}_n$, belong to the equivalence class $\mathcal{U}_1$. 
\end{lemma}

\begin{proof}
Since $G$ is a quasigroup, for every permutation $W \in \mathcal{W}$ and $a\in \mathcal{I}_n$ there are permutations $W', W'' \in \mathcal{W}$ such that $W * W' = \textbf{a}$ and $\textbf{a} * W'' = W$.  In other words,  the quasigroup  $G = G[1]$ contains $W$-diagonals of type $\textbf{a}$ and $\textbf{a}$-diagonals  of type $W$. Therefore, $\textbf{a}$ and $W$ belong to the same equivalence class and the period of this class is equal to $1$ or $2$.
\end{proof}

For the equivalence class $\mathcal{U}_1$ we can refine the general estimation on the size from Corollary~\ref{unitsizes} with the help of the following quasigroup invariant.

Given a binary quasigroup $G$, define $P^k (G)$ to be the set of all elements of $G$ that allow a factorization (ordered and bracketed in any way) containing every element of $G$ precisely $k$ times. For shortness, let $P(G) = P^1(G)$.

For all $k \in \mathbb{N}$ it holds $P^k(G) \subseteq P^{k+2} (G)$, since for each $g \in P^k(G)$ there is $h \in G$ such that $g * h = g$ and $h$ can be presented as a product $f_1 * f_2$ exactly $n$ times, where $f_1$ ranges over the elements of $G$ as
$f_2$ also ranges over the elements of $G$.
Note that there are examples where $P^k(G) \not\subseteq P^{k+1} (G)$, e.g., when $G$ is $\mathbb{Z}_4$.
 
Let us define the set $P^{\infty} (G) = \bigcup\limits_{i=1}^{\infty} P^i (G)$. It is well defined due to $P^k(G) \subseteq P^{k+2} (G)$ and $G$ is a finite quasigroup.  It is easy to check that $P^{\infty} (G)$ is a subquasigroup in $G$. 

For a tuple $V \in \mathcal{I}_n^k$, $V = (v_1, \ldots, v_k)$, $v_i \in G$, let
 $$\Pi(V) = (\cdots(v_1 * v_2) * \cdots * v_{k-1}) * v_k.$$

From the definitions of sets $P^\infty(G)$ and $\Pi(V)$, we have the following property.

\begin{utv} \label{U1prod}
Let $G$ be a quasigroup of order $n$.  Then for every tuple $V \in \mathcal{U}_1$ it holds $\Pi(V) \in P^\infty(G)$. 
\end{utv}

\begin{proof}
It is sufficient to note that for  every  $V \in \mathcal{U}_1$ there is a collection of permutations $(W_0, \ldots, W_d)$ such that $V = (\cdots (W_0 * W_1) * \cdots *W_{d-1}) * W_d $.
\end{proof}

From this statement, we have the following bound on the size of $\mathcal{U}_1$.

\begin{sled} \label{U1size}
Given a binary quasigroup $G$ of order $n$, the cardinality of the equivalence class $\mathcal{U}_1$ is $p n^{n-1}$, where  $1 \leq p \leq |P^{\infty} (G)|$.
\end{sled}

\begin{proof}
By Proposition~\ref{Vinunits} and Corollary~\ref{unitsizes}, the size of the equivalence class $\mathcal{U}_1$ is equal to $p n^{n-1}$, where $p$ is the number of tuples from the sets $V^*_j$ in the class $\mathcal{U}_1$. For a given $V \in \mathcal{I}_n^n$ there are exactly $|P^\infty (G)|$ tuples $U \in V^*_j$ with $\Pi(U) \in P^\infty (G)$.  By Proposition~\ref{U1prod}, for every $U \in \mathcal{U}_1$ we have $\Pi(U) \in P^\infty (G)$,  therefore the constant $p$ is not greater than $|P^\infty (G)|$.
\end{proof}

Unfortunately, the size of the class $\mathcal{U}_1$ for a quasigroup $G$ is not defined by $|P^\infty(G)|$ and we cannot distinguish tuples $U$ and $V$ from different equivalence classes by the sets $\Pi(U)$ and $\Pi(V)$.

\textbf{Example 2.}
Consider a binary quasigroup $G$ of order $4$ with the Cayley table
$$\begin{array}{c|cccc}
*  & 1 & 2 & 3 & 4 \\
\hline
1 & 2 & 1 & 4 & 3 \\
 2 & 1 & 2 & 3 & 4 \\
3 & 3 & 4 & 1 & 2 \\
4 & 4 & 3 & 2 & 1 
\end{array}$$

It is easy to see that $G$ is isotopic to the group $\mathbb{Z}_2^2$ by the means of transposition of the first two rows. 

By Lemma~\ref{classdescr} from the next section, the group $\mathbb{Z}_2^2$ has four equivalence classes $\mathcal{U}_1, \ldots, \mathcal{U}_4$ of period $1$, where $\mathcal{U}_i = \{ V |\Pi(V) = i  \}.$
With the help of Lemma~\ref{tranitiso}, we obtain that the quasigroup $G$ has three equivalence classes: $\mathcal{U}'_1 = \mathcal{U}_1$, $\mathcal{U}'_2 = \mathcal{U}_2$, and $\mathcal{U}'_3 = \mathcal{U}_3 \cup \mathcal{U}_4$. Here classes $\mathcal{U}'_1$ and $\mathcal{U}'_2$ have period $1$, and the class $\mathcal{U}'_3$ has period $2$.

It can be verified directly that there are tuples $U \in \mathcal{U}'_1$ and $V \in \mathcal{U}'_2$ (e.g. $U = (1,2,3,4)$ and $V = (1,1,3,4)$) such that $\bigcup\limits_{\pi \in S_4}\Pi(U^\pi) = \bigcup\limits_{\pi \in S_4}\Pi(V^\pi) =P^\infty(G) = \{ 1,2 \}$ in the quasigroup $G$.

Summarizing obtained results, we prove the main theorem of the present section that describes diagonals in a general iterated quasigroup.

\begin{teorema} \label{quasigeneral}
Let $G$ be a binary quasigroup of order $n$, $U, V \in \mathcal{I}_n^n$, and  $\mathcal{U}_1, \ldots, \mathcal{U}_{m}$  be the  partition of $\mathcal{I}_n^n$ into equivalence classes of periods $\tau_1, \ldots, \tau_m$, respectively. Suppose that all units of a class $\mathcal{U}_i$ have sizes $r_i n^{n-1}$, $1 \leq r_i \leq n$. Then the following hold:
\begin{itemize}
\item If tuples $U$ and $V$ belong to different equivalence classes, then for all $d \in \mathbb{N}$ there are no $U$-diagonals of type $V$ in the $d$-iterated quasigroup $G[d]$.
\item If both $U$ and $V$ belong to the same equivalence class $\mathcal{U}_i$, then there is $0 \leq k \leq \tau_i - 1$ such that for all $d$ greater than some $d_0 \in \mathbb{N}$ the number $T_{U,V} (d)$ of $U$-diagonals of type $V$ in the $d$-iterated quasigroup $G[d]$ is 
$$T_{U,V} (d) = \frac{n!^{d}}{r_i n^{n-1}}   \cdot(1 + o(1)) $$
if $d \equiv k \mod \tau_n$ and $T_{U,V} (d) = 0$ otherwise. 
\end{itemize}
\end{teorema}

\begin{proof}
Theorems~\ref{PFform} and~\ref{PerFrob} give the partition into equivalence classes and units. By Theorem~\ref{PFform}, there are no $U$-diagonals of types $V$ in $G[d]$ if $U$ and $V$ are from different equivalence classes. By Theorem~\ref{PerFrob} and definition of units, there are no $U$-diagonals of types $V$ in $G[d]$ if both $U$ and $V$ are from unit $\mathcal{U}_i$ of period $\tau_i$, $U \in \mathcal{Y}_k^i$, $V \in \mathcal{Y}_l^i$ but $d \not\equiv k-l \mod \tau_i$. Otherwise, Theorem~\ref{PerFrob} states that for large $d$ the number of $U$-diagonals of types $V$ in $G[d]$ is close to $ \frac{1}{M_i} n!^d$, where $M_i$ is the size of units in the equivalence class $\mathcal{U}_i$, for which $U,V \in \mathcal{U}_i$. The sizes of units were estimated in Corollary~\ref{unitsizes}.
\end{proof}


\subsection{Diagonals in iterated loops and other special quasigroups}

In this section, we consider how additional conditions on a quasigroup $G$ affect  equivalence classes and their units. We start with quasigroups having the right inverse-property. 

A quasigroup $(G,*)$ is said to have the \textit{right inverse-property} if there is a permutation $\pi $ of the set $G$  such that $(g * h) * \pi (h) = g$ for all $g, h \in G$.

\begin{utv}
Let $G$ be a binary quasigroup with the right-inverse property. Then each equivalence class $\mathcal{U}_i$ has period $1$ or $2$.
\end{utv}

\begin{proof}
By Proposition~\ref{canoninunits}, each equivalence class $\mathcal{U}_i$  of $G$ contains tuples of the form $\textbf{a}_j(b)$. Since $G$ has the right inverse-property, for every permutation $W \in \mathcal{W}$  there is  a permutation $W' \in \mathcal{W}$ such that $(\textbf{a}_j(b) * W) * W' = \textbf{a}_j(b)$. Thus the period of the tuple $\textbf{a}_j (b)$ from $\mathcal{U}_i$ is not greater than $2$.
\end{proof}

The right-inverse property (and the complementary left-inverse property) was introduced in book~\cite{belous.quasigr}. If a quasigroup $G$ has the left and right inverse-properties, then $G$ is said to be a quasigroup with the inverse property or an IP-quasigroup. For more information on inverse properties of quasigroups see~\cite{KeedSher.invprop}.

Next, we consider the case of loops. A quasigroup $G$ is said to be a \textit{loop} if there is the \textit{identity element} $e \in G$ such that for each $g \in G$ it holds $e * g = g * e = g$.

Given a loop $G$, a normal subloop of $G$ is defined in the same way as for a general algebraic system. 
Let the \textit{commutator subloop} $G'$ of a loop $G$ be the smallest normal subloop of $G$ such that $G / G'$ is an abelian group.

In~\cite{pula.prodquasi} it was proved a number of properties of sets $P^k(G)$ and $P^\infty(G)$,  when $G$ is a loop. For example, the set $P^\infty (G)$ is a subloop of $G$ and  for every $k \in \mathbb{N}$ the set $P^k(G)$ is contained in a single coset of $G'$. Moreover, if $P^{\infty} (G)$ is a normal subloop, then $P^{\infty} (G) = G'$ or $P^{\infty} (G) = G' \cup g G'$, where $g* g \in G'$.

These results give the following improvement on the sizes of units in the equivalence class $\mathcal{U}_1$.

\begin{utv} \label{loopimpr}
If $G$ is a loop of order $n$, then the size of units in the equivalence class $\mathcal{U}_1$ is equal to $r n^{n-1}$ for some integer $r$, $1 \leq r \leq |G'|$.
\end{utv}

\begin{proof}
By definitions, there is $k \in \mathbb{N}$ such that for every tuple $V$ from a unit $\mathcal{Y}$ of the class $\mathcal{U}_1$ it holds $\Pi(V) \in P^k(G)$. Since $P^k(G)$ is contained in a single coset of $G'$, we have $|P^k(G)| \leq |G'|$. Acting similar to the proof of Corollary~\ref{U1size}, we obtain the required statement.
\end{proof}

We don't know how to efficiently find  equivalence classes and units and estimate their sizes for a general binary quasigroup $G$. It seems probable that in this case they cannot be expressed in terms of some algebraic invariant and mostly depend on the structure of the Cayley table of the quasigroup. To illustrate this, consider the following example.

\textbf{Example 3.} Let $(G,*)$ be a binary quasigroup of order $n = 2k$ with the Cayley table of the form
$$
\begin{array}{cc}
A' & B' \\ B'' & A''
\end{array}
$$ 
where $A', A''$ are arbitrary latin squares of order $k$ under the symbol set $S_1 = \{ 1, \ldots, k \}$ and $B', B''$ are latin squares under the symbol set $S_2 = \{ k+1, \ldots, 2k \}$.
By the definitions, every permutation $W \in \mathcal{W}$ contains the same number of entries from $S_1$ and $S_2$. Moreover, for all $a,b \in \mathcal{I}_n$ we have $a * b \in S_2$ if $a$ and $b$ belong to different symbol sets, and $a * b \in S_1$ otherwise.

We will say that a tuple $V \in \mathcal{I}_n^n$ is even if $V$ contains an even number of symbols from each of the sets $S_1$ and $S_2$, and that $V$ is odd otherwise. It is easy to check that if $k$ is even ($n \equiv 0 \mod 4$), then for every permutation $W \in \mathcal{W}$ the tuple $V * W$ has same parity as $V$, and if $k$ is odd ($n \equiv 2 \mod 4$), then $V * W$ has the different parity.

Therefore, for $n \equiv 0 \mod 4$ the quasigroup $G$ has at least two different equivalence classes (composed of odd or even tuples), and for $n \equiv 2 \mod 4$ each (possibly unique) equivalence class of $G$ contains at least two units.

\subsection{Diagonals in iterated groups}

Thanks to the association property, for a given group $G$ we can completely describe the structure of equivalence classes and their units.
In particular,  we will see that the variety of diagonals in an iterated group depends on the fulfillment of the Hall--Paige condition.
We also refine the statement of Theorem~\ref{quasigeneral} for groups and prove that for large $d$ the number of diagonals in $d$-iterated groups $G[d]$ is given by the size of the commutator $G'$.

Recall that $G$ is a \textit{Hall--Paige group} if all its Sylow $2$-subgroups are trivial or non-cyclic ($G$ satisfies the condition of Theorem~\ref{HPconj}).  

Given a group $G$, the \textit{commutator} subgroup $G'$ is  the smallest normal subgroup of $G$ such that $G / G'$ is an abelian group. Equivalently, $G'$ is  a subgroup generated by commutators $ghg^{-1}h^{-1}$, $g,h \in G$. The group $H = G / G'$ is known as the abelianization  of $G$. Since $G'$ is a normal subgroup, the group $G$ is partitioned into cosets $hG'$, $h \in H$. 

The following result of D\'enes and Hermann~\cite{DenHer.grprod} connects the commutator subgroup $G'$ with the set $P(G)$. Recall that here $P(G)$ denotes the set of products of all elements of $G$: $P(G) = \{ \Pi(W) | W \in \mathcal{W} \}$.

\begin{teorema}[\cite{DenHer.grprod}] \label{DHth}
Let $G$ be a group. Then either $P(G) = G'$ or $P(G) = gG'$, where $g$ is the unique element of order $2$ in a nontrivial, cyclic Sylow $2$-subgroup of $G$.
\end{teorema}

For our purposes, we need several approaches to the definition of Hall--Paige groups. Some requirements, which are equivalent to the Paige--Hall condition, are presented in the following table.

\begin{center}
\begin{tabular}{|c|l|l|}
\hline
~ & $G$ is a Hall--Paige group  & $G$ is a non-Hall--Paige group \\
\hline  \hline
(1) & all Sylow $2$-subgroups of $G$ & there exists a nontrivial   \\
~  &  are trivial or non-cyclic &  cyclic  Sylow $2$-subgroup in $G$ \\
\hline
(2) & the Cayley table of $G$ & the Cayley table of $G$ \\
~ & has a transversal & has no transversals \\
\hline
(3) & $P(G) = G'$ & $P(G) = gG'$  \\
~ & ~ &  for some $g \notin G'$, $g^2 \in G'$ \\
\hline
(4) & all $d$-iterated groups $G[d]$  &  $d$-iterated groups $G[d]$ have   \\
~ & have transversals &  transversals only if $d$ is even \\
\hline
\end{tabular}
\end{center}

\begin{teorema}
Conditions $(1)$--$(4)$ are equivalent.  Every condition can be taken as a definition of a Hall--Paige group (a non-Hall--Paige group).
\end{teorema}

\begin{proof}

(1) $\Leftrightarrow$ (2): It is the Hall--Paige conjecture  (proved in~\cite{wilcox.HPred} and~\cite{evans.admsporgr}).

(3): This alternative is Theorem~\ref{DHth}, proved in \cite{DenHer.grprod}.

(1) $\Leftrightarrow$ (3): By~\cite{DenKeed.gradm}, the Hall--Paige condition (1)  is equivalent to $e \in P(G)$. It is possible only when $P(G) = G'$. A short proof of (1) $\Leftrightarrow$ (3) can be also found in~\cite{LeeWan.LShallpaige}.

(3) $\Leftrightarrow$ (4): It will be proved in Theorem~\ref{grouptransintro}.

\end{proof}

In order to describe equivalence classes and units for iterated groups,  we need the following refinement of Proposition~\ref{toconst}, controlling the resulting tuples.

\begin{utv} \label{toconstgr}
Let $G$ be a group of order $n$ and let $e$ be the identity element of $G$. Then for each tuple $V \in \mathcal{I}_n^n$, $V = (v_1, \ldots, v_n)$, there is $l \in \mathbb{N}$, $l \leq n$ for which  there exists a $V$-diagonal of type  $\textbf{e}_1(\Pi(V))$ in the iterated group $G[2l]$.
\end{utv}

\begin{proof}
Suppose that the last $n-k$ positions of $V$ are equal to $e$ and the last element of $V$ that is distinct from $e$ is located in the $k$-th position.
The proof goes by the induction on $k$. For the base of induction ($k =0$ and $k=1$) the statement is true with $l=0$. 

Assume that $k \geq 2$. 
Consider permutations $W, W' \in \mathcal{W}$, $W = (w_1, \ldots, w_n)$, $W' = (w'_1, \ldots, w'_n)$, for which the following equalities hold
\begin{gather*} 
v_1 * w_1 * w'_1 = v_1; \\
\vdots \\
v_{k-2} * w_{k-2} * w'_{k-2} = v_{k-2}; \\
v_{k-1} * w_{k-1} * w'_{k-1} = v_{k-1} * v_{k}; \\
v_{k} * w_{k} * w'_{k} = e; \\
e * w_{k+1} * w'_{k+1} = e; \\
\vdots \\
e * w_{n} * w'_{n} = e. 
\end{gather*} 
These equalities mean that $V*W *W' = V'$ for some tuple $V'$, whose first $k - 2$ positions are the same as in tuple $V$, $\Pi(V') = \Pi(V)$ and  the last  $n - k+1$ elements of $V'$ are equal to $e$. 

Let us show that the required permutations $W$ and $W'$ exist.  Let $g$ be an arbitrary element of $G$. Put 
$$w_k = g; ~~~~ w'_{k} = g^{-1} * v_{k}^{-1};  ~~~~ w_{k-1} = v_{k} * g; ~~~~ w'_{k-1} = g^{-1}. $$
 Note that the condition $v_k \neq e$ implies that $w_{k-1} \neq w_k$ and $w'_{k-1} \neq w'_k$. It is easy to see that $(k-1)$-th and $k$-th required equalities are satisfied by these choices.

Let the other $n-2$ pairs of elements $(w_1, w'_1)$, \ldots, $(w_{k-2}, w'_{k-2})$, $(w_{k+1}, w'_{k+1})$, \ldots, $(w_n, w'_n)$ be all the pairs $(w,w^{-1})$ of mutually inverse elements of the group $G$, except for the pairs $(g, g^{-1})$ and $(v_k *g, g^{-1} * v_k^{-1})$, whose elements have been already used in $W$ and $W'$.

Thus, we construct a $V$-diagonal of type $V'$ in the iterated group $G[2]$. By the inductive assumption, there exists a $V'$-diagonal of type $\textbf{e}_1(\Pi(V'))$ in the iterated group $G[2l-2]$ for some $l$. Since $\Pi(V') = \Pi(V)$,  Lemma~\ref{diagiter} implies that there is a $V$-diagonal of type $\textbf{e}_1(\Pi(V))$ in $G[2l]$.
\end{proof}

\begin{lemma} \label{subgroupU1}
Let $G$ be a group of order $n$. Then the equivalence class $\mathcal{U}_1$ is a subgroup in the $n$-th Cartesian power $G^n$. Moreover, each equivalence class $\mathcal{U}_i$, $i = 1, \ldots, m$, is a left coset of $\mathcal{U}_1$ in $G^n$.
\end{lemma}

\begin{proof}
By Lemma~\ref{transclass}, the identity element $\textbf{e}$ of the group $G^n$ belongs to the class $\mathcal{U}_1$. By the definitions, every tuple $V \in \mathcal{U}_1$  is a product of some permutations $W \in \mathcal{W}$. Thanks to associativity, if  tuples $U$ and $V$ belong to $\mathcal{U}_1$, then $U * V \in \mathcal{U}_1$. At last, for every $V \in \mathcal{U}_1$, $V = W_1 * \cdots * W_k$ its inverse $V^{-1} = W_k^{-1} * \cdots * W_1^{-1}$ also belongs to the class $\mathcal{U}_1$. Thus $\mathcal{U}_1$ is a subgroup in $G^n$.

By definition of equivalence classes, each $\mathcal{U}_i$ consists of tuples constructed by the means of consecutive right multiplications of a tuple by permutations from $\mathcal{W}$. Thanks to the associativity of the multiplication in the group $G^n$,  the class $\mathcal{U}_i$ is a left coset of $\mathcal{U}_1$.
\end{proof}

\begin{sled}
If $G$ is a group, then each equivalence class $\mathcal{U}_i$ has the same size.
\end{sled}

We are interested to know if it is true that for any quasigroup $G$ the equivalence class $\mathcal{U}_1$ is a subquasigroup in $G^n$. Note that in a general case there are no equalities between sizes of equivalence classes (see Example 2).

Now we are ready to describe the equivalence class $\mathcal{U}_1$ for iterated groups.  For abelian groups a similar result was established by Hall in~\cite{hall.abelperm}.

\begin{lemma}\label{U1descr}
Let $G$ be a group of order $n$ and $G'$ be the commutator subgroup of $G$. 
\begin{enumerate}
\item Assume that $G$ is a Hall--Paige group. Then the equivalence class $\mathcal{U}_1$ consists of a single unit. A tuple $V$ belongs to $\mathcal{U}_1$ if and only if $\Pi(V) \in  G'$. In particular, $|\mathcal{U}_1| = |G'|n^{n-1}$.
\item Assume that  $G$ is a non-Hall--Paige group and $g$ is the element of order $2$ from a nontrivial, cyclic Sylow $2$-subgroup of $G$. Then the equivalence class $\mathcal{U}_1$ consists of two units $\mathcal{Y}_1$ and $\mathcal{Y}_2$. A tuple $V$ belongs to the unit $\mathcal{Y}_1$ if and only if $\Pi(V) \in gG'$ and $V$ belongs to $\mathcal{Y}_2$ whenever $\Pi(V) \in G'$. The cardinality  of each unit  $\mathcal{Y}_i$ is $|G'| n^{n-1}$.
\end{enumerate}
\end{lemma}

\begin{proof}
1. Let $G$ be a Hall--Paige group of order $n$. By condition (2) of the definition of Hall--Paige groups, the Cayley table of $G$ has a transversal. In other words, there exists a $\mathbb{W}$-diagonal of type $W$ for some permutation $W \in \mathcal{W}$ in the iterated group $G[1]$.  So there are two consecutive units in $\mathcal{U}_1$ containing permutations. Meanwhile, Proposition~\ref{permclosed} claims that all permutations are contained in a single unit. Therefore, the equivalence class $U_1$ has a unique unit.

For every tuple $V \in \mathcal{U}_1$ it holds $\Pi(V) \in G'$, because, by Proposition~\ref{U1prod}, $\Pi(V) \in P^\infty(G)$ and the condition (3)  in the definition of Hall--Paige groups gives that $P(G) = G' = P^{\infty} (G)$.

We prove that every tuple $V$ with property $\Pi(V) \in G'$ belongs to the class $\mathcal{U}_1$ by comparing cardinalities of these sets. It is obvious that the number of tuples $V$ for which $\Pi(V) \in G'$ is equal to $|G'| n^{n-1}$.

Let us estimate the size of the class $\mathcal{U}_1$. Consider an arbitrary permutation $W \in \mathcal{W}$ from the equivalence class $\mathcal{U}_1$ (see Lemma~\ref{transclass}). By Proposition~\ref{toconstgr}, the tuple  $\textbf{e}_1(\Pi(W))$ also  belongs to the class $\mathcal{U}_1$. Therefore, for every $b \in P(G)$ we have $\textbf{e}_1(b) \in \mathcal{U}_1$. Since $P(G) = G'$ by Theorem~\ref{DHth},  there are at least $|G'|$ different tuples  of the form $\textbf{e}_1 (b)$ in the class $\mathcal{U}_1$. Finally, Proposition~\ref{Vinunits} implies that the equivalence class $\mathcal{U}_1$ contains at least $|G'| n^{n-1}$ tuples.

2. Assume now that $G$ is a non-Hall--Paige group of order $n$. By condition (3) of the definition of non-Hall--Paige groups,  we have that $P(G) = gG'$ for some $g \in G$ of order $2$. From  the definition of  the equivalence class $\mathcal{U}_1$ it follows that every $V \in \mathcal{U}_1$ can be presented as a product $ W_1 * \cdots * W_k$ for some collection of permutations $W_i \in \mathcal{W}$ and $k \in \mathbb{N}$. If $k$ is even, then $\Pi(V) \in G'$,  and $\Pi(V) \in gG'$ otherwise. Therefore, the equivalence class $\mathcal{U}_1$ has at least two units $\mathcal{Y}_1$ and $\mathcal{Y}_2$ that are contained in sets $\{ V | \Pi(V) \in gG' \}$ and  $\{ V | \Pi(V) \in G' \}$ respectively. But by Lemma~\ref{transclass}, the number of units in $\mathcal{U}_1$ is not greater than $2$.
 
To prove the equalities $\mathcal{Y}_1 = \{ V | \Pi(V) \in gG' \}$ and $\mathcal{Y}_2 = \{ V | \Pi(V) \in G' \}$, we compare their cardinalities by the same way as in the previous clause.

\end{proof}

For groups, we are able to describe not only tuples in the equivalence class $\mathcal{U}_1$ but all other equivalence classes and their units.

\begin{lemma} \label{classdescr}
Let $G$ be a group of order $n$ and $G'$ be the commutator subgroup of $G$. 
\begin{enumerate}
\item If $G$ is a Hall--Paige group, then each equivalence class $\mathcal{U}_i$ contains a single unit and there are some $h_i \in G$ such that $\mathcal{U}_i = \{V | \Pi(V) \in h_iG' \}$.
\item Assume that $G$ is a non-Hall--Paige group and $g$ is the element of order $2$ from a nontrivial, cyclic Sylow $2$-subgroup of $G$. Then each equivalence class $\mathcal{U}_i$ consists of two units $\mathcal{Y}^i_1$ and $\mathcal{Y}^i_2$ and there are some $h_i \in G$ such that $\mathcal{Y}_1^i = \{V | \Pi(V) \in h_igG' \}$ and $\mathcal{Y}_2^i = \{V | \Pi(V) \in h_iG' \}$.
\end{enumerate}
\end{lemma}

\begin{proof}
By Lemma~\ref{subgroupU1}, each equivalence class $\mathcal{U}_i$ is a left coset of the class $\mathcal{U}_1$. Therefore, all classes $\mathcal{U}_i$ have the same size and the same number of units as the equivalence class $\mathcal{U}_1$, for which these quantities were found in Lemma~\ref{U1descr}.

Consider a tuple $V$ from an equivalence class $\mathcal{U}_i$ and assume that $\Pi(V)$ belongs to some coset $h_i G'$. By the definition of equivalence classes, for every other tuple $U$ from the class $\mathcal{U}_i$  there is some collection of permutations $(W_1, \ldots, W_d)$ such that $U = V * W_1 * \cdots * W_d$. 
Using the condition (3) of the definition of Hall--Paige groups, we see that if $G$ is a Hall--Paige group, then $\Pi(U)$ belongs to the same coset $h_i G'$. In the case when $G$ is a non-Hall--Paige group, if $d$ is even then $\Pi(U)$ belongs to the coset $h_i G'$ and if $d$ is odd then $\Pi(U) \in h_igG'$.

The equalities $\mathcal{U}_i = \{V | \Pi(V) \in h_iG' \}$ for Hall--Paige groups and $\mathcal{Y}_1^i = \{V | \Pi(V) \in h_igG' \}$ and $\mathcal{Y}_2^i = \{V | \Pi(V) \in h_iG' \}$ for non-Hall--Paige groups follow from equalities between the cardinalities of these sets.  
\end{proof}

At last, we are ready to prove the general result on diagonals in iterated groups.

\begin{teorema} \label{groupgeneral}
Let $G$ be a group of order $n$ with the commutator subgroup $G'$ and let $U,V \in \mathcal{I}_n^n$.
\begin{enumerate}
\item Assume that $G$ is a Hall--Paige group. 
\begin{itemize}
\item If $\Pi(U)$ and $\Pi(V)$ belong to different cosets of $G'$, then there are no $U$-diagonals of type $V$ in all $d$-iterated groups $G[d]$. 
\item If $\Pi(U)$ and $\Pi(V)$ belong to the same coset of $G'$, then there is some $d_0$ such that for all $d \geq d_0$ the $d$-iterated groups $G[d]$ have $U$-diagonals of type $V$. The number $T_{U,V} (d)$ of $U$-diagonals of type $V$ in $G[d]$ is asymptotically
$$T_{U,V} (d) = \frac{n!^{d}}{|G'| n^{n-1}}   \cdot (1 + o(1)). $$
\end{itemize}
\item Assume that $G$ is a non-Hall--Paige group and $g$ is the unique element of order $2$ from a nontrivial, cyclic Sylow $2$-subgroup of $G$.
\begin{itemize}
\item If  $\Pi(U)$ and $\Pi(V)$ are not from the union of cosets $hG' \cup hgG'$ for some $h \in G$, then there are no $U$-diagonals of type $V$ in all $d$-iterated groups $G[d]$. 
\item There is $d_0 \in \mathbb{N}$ such that the following hold.  If one can find $h \in G$ for which $\Pi(U), \Pi(V) \in hG'$, then the $d$-iterated groups $G[d]$ have $U$-diagonals of type $V$ for all even $d \geq d_0$, and if $\Pi(U) \in h G'$, $\Pi(V) \in gh G'$ for some $h \in G$, then the $d$-iterated groups $G[d]$ have $U$-diagonals of type $V$ for all odd $d \geq d_0$. Otherwise $G[d]$ has no $U$-diagonals of type $V$. If $d$ has an appropriate parity, then the number $T_{U,V} (d)$ of $U$-diagonals of type $V$ in $G[d]$ is asymptotically
$$T_{U,V} (d) = \frac{n!^{d}}{|G'| n^{n-1}}   \cdot (1 + o(1)).$$
\end{itemize}
\end{enumerate}
\end{teorema}

\begin{proof}
The result follows from Theorem~\ref{quasigeneral}, which describes how the number of $U$-diagonals of type $V$ is connected with the structure of equivalence classes and units, and from Lemma~\ref{classdescr}, where these sets were characterized for groups.
\end{proof}

\section{Proofs of the main results}

We start with a proof of the theorem on transversals in iterated quasigroups. 

\begin{proof}[Proof of Theorem~\ref{quasitransintro}]
Recall that a transversal in a multiary quasigroup is a $\mathbb{W}$-diagonal of type $W$, where $W$ is a  permutation. Given a quasigroup of order $n$, there are exactly $n!$ possible types of diagonals corresponding to transversals. 

Let us fix some permutation $W \in \mathcal{W}$.  By Proposition~\ref{permclosed} and Lemma~\ref{transclass}, the permutations $\mathbb{W}$ and $W$ belong to the same unit of the equivalence class $\mathcal{U}_1$ and the period of the class $\mathcal{U}_1$ is not greater than $2$.  By Proposition~\ref{loopimpr}, if $G$ is a loop then the size of units in the equivalence class $\mathcal{U}_1$ is $rn^{n-1}$, where $1 \leq r \leq |G'|$.

To count the number of $\mathbb{W}$-diagonals of type $W$ in an iterated quasigroup $G[d]$, it only remains to apply Theorem~\ref{quasigeneral}.
\end{proof}

\begin{proof}[Proof of Theorem~\ref{grouptransintro}]
The proof is similar to the proof of Theorem~\ref{quasitransintro}, but instead of Theorem~\ref{quasigeneral} we use Theorem~\ref{groupgeneral}.

If $G$ is a Hall--Paige group, then all $d$-iterated groups have transversals, because, by condition (2) of the definition of Hall--Paige groups, $G[1]$ has a transversal and, by the proof of Lemma~\ref{transclass}, for each permutation $W \in \mathcal{W}$ there exists a $W$-diagonal of type $W$ in $G[2]$.
\end{proof}

At last, we prove a similar result for near transversals.

\begin{proof}[Proof of Theorem~\ref{nearintro}]
In our terms, a near transversal in a quasigroup of order $n$ is an arbitrary $\mathbb{W}$-diagonal of type $W_i(b)$, where $W \in \mathcal{W}$ is a permutation and $i,b\in \mathcal{I}_n$.

By Proposition~\ref{Vinunits}, for every permutation $W \in \mathcal{W}$ and $i \in \mathcal{I}_n$ the number of tuples from the set $W_i^*$ in each unit of the class $\mathcal{U}_1$ is equal to some $r \geq 1$.  It means that there is $d_0 \in \mathbb{N}$ such that  for all $d \geq d_0$ every $d$-iterated quasigroup $G[d]$ has near transversals.

If Brualdi's conjecture is true, then for all quasigroups $G$ there is a $\mathbb{W}$-diagonal of type $W_i(b)$ in  $G[1]$.  Using Lemma~\ref{transclass}, we have that in this case all $d$-iterated quasigroups have near transversals. For groups, Brualdi's conjecture is true by~\cite{GodHal.grnear}.

Assume that a tuple $W_i(b)$, $W \in \mathcal{W}$, belongs to a unit $\mathcal{Y}$ of the equivalence class $\mathcal{U}_1$. If the tuple $W_i(b)$ coincides with the permutation $W$, then, by Proposition~\ref{permclosed}, the unit $\mathcal{Y}$ contains all $n!$ permutations from $\mathcal{W}$. If the tuple $W_i(b)$ is not a permutation, then Proposition~\ref{permclosed} and permutations of coordinates of the tuple $W_i(b)$ give $\frac{n}{2} n!$ different tuples in the unit $\mathcal{Y}$. 

Therefore, if a unit $\mathcal{Y}$ of the equivalence class $\mathcal{U}_1$ contains permutations, then there are  exactly $(\frac{n}{2}(r-1) +1)n!$ tuples of the form $W_i(b)$ in $\mathcal{Y}$, otherwise, the number of such tuples in $\mathcal{Y}$ is $\frac{n}{2} r n!$.

The resulting asymptotic of the number of near transversals follows from Theorems~\ref{quasigeneral} and~\ref{groupgeneral}.
\end{proof}

\section{Concluding remarks}

There are many ways to extend the technique described in the present paper.
First of all, in one of our subsequent papers we are going to consider diagonals in a general composition of multiary quasigroups and connect them with the permanents and contractions of multidimensional matrices. In particular, we study the composition not only of quasigroups and latin hypercubes but other stochastic multidimensional arrays.

Another direction of future work is finding explicit formulas for the numbers of transversals and diagonals of other types in some iterated quasigroups of small orders or arity.

At last, in paper~\cite{my.iter}, the author estimated the asymptotic numbers of partial diagonals,  plexes, and multiplexes in iterated quasigroups. Using similar methods, these bounds can be significantly refined. Moreover, it is possible to generalize our approach to other structures in latin hypercubes, for example, to subcubes or trades.

\section*{Acknowledgements}
The work was carried out within the framework of the state contract
of the Sobolev Institute of Mathematics (project no. 0314-2019-0016) and supported in part by the Young Russian Mathematics award.  The author is grateful to anonymous referees for the correction of many small mistakes and inaccuracies in the manuscript.

\end{document}